\documentclass[12pt, reqno]{amsart}
\usepackage{graphicx, cite, epsfig}
\usepackage{amsfonts,amsmath,amssymb}
\usepackage{theorem}

\columnwidth\textwidth

\newcounter{theorem}
\def\openthm#1#2{\refstepcounter{theorem}\bigskip

{\noindent\bf#1~\thetheorem\if#2!{. }\else{ (#2).}\fi}
\it}

\def\thmskip{}
\newenvironment{theorem}[1][!]{\openthm{Theorem}{#1}}{\thmskip}
\newenvironment{lemma}[1][!]{\openthm{Lemma}{#1}}{\thmskip}
\newenvironment{corollary}[1][!]{\openthm{Corollary}{#1}}{\thmskip}
\newenvironment{proposition}[1][!]{\openthm{Proposition}{#1}}{\thmskip}
\newenvironment{conjecture}[1][!]{\openthm{Conjecture}{#1}}{\thmskip}

\newcounter{remark}
\def\openrem#1#2{\refstepcounter{remark}\bigskip
{\noindent \it \bfseries#1~\theremark\if#2!{. }\else{ (#2). }\fi}}
\newenvironment{remark}[1][!]{\openrem{Remark}{#1}}{\qed}

\def\ilist{\renewcommand{\labelenumi}{(\roman{enumi})}}

\def\R{\mathbb{R}}

\def\Z{\mathbb{Z}}

\def\CC{{\rm C}}

\def\pp{\partial}

\def\<{\langle}
\def\>{\rangle}

\def\eps{\varepsilon}

\def\E{\mathcal{E}}

\def\Fs{\mathcal{F}}

\def\Ys{\mathcal{Y}}
\def\Us{\mathcal{U}}
\def\As{\mathcal{A}}
\def\Cs{\mathcal{C}}
\def\Ps{\mathcal{P}}

\def\vb{v}
\def\wb{w}

\def\zb{z}

\DeclareMathOperator{\cond}{cond}

\renewcommand{\cases}[1]{\left\{ \begin{array}{ll} #1 \end{array} \right.}

\newcommand{\smfrac}[2]{{\textstyle \frac{#1}{#2}}}
\newcommand{\lpnorm}[2]{\left\|#1\right\|_{\ell^{#2}_\eps}}
\def\half{\smfrac{1}{2}}

\begin{document}

\title[Sharp Stability Estimates for QCF Methods]{Sharp Stability
  Estimates\\ for the Force-based Quasicontinuum Method}

\author{M. Dobson}
\author{M. Luskin}
\address{School of Mathematics, 206 Church St. SE,
  University of Minnesota, Minneapolis, MN 55455, USA}
\email{dobson@math.umn.edu, luskin@umn.edu}

\author{C. Ortner} \address{Mathematical Institute, St. Giles' 24--29,
  Oxford OX1 3LB, UK} \email{ortner@maths.ox.ac.uk}

\date{\today}

\keywords{atomistic-to-continuum coupling, quasicontinuum method,
  sharp stability estimates}

\thanks{ This work was supported in part by DMS-0757355, DMS-0811039,
  the Department of Energy under Award Number DE-FG02-05ER25706, the
  Institute for Mathematics and Its Applications, the University of
  Minnesota Supercomputing Institute, the University of Minnesota
  Doctoral Dissertation Fellowship, and the EPSRC critical mass
  programme ``New Frontier in the Mathematics of Solids.''  }

\subjclass[2000]{65Z05,70C20}

\begin{abstract}
  A sharp stability analysis of atomistic-to-continuum coupling
  methods is essential for evaluating their capabilities for
  predicting the formation and motion of lattice defects. We formulate
  a simple one-dimensional model problem and give a detailed analysis
  of the stability of the force-based quasicontinuum (QCF) method. The
  focus of the analysis is the question whether the QCF method is able
  to predict a critical load at which fracture occurs.

  Numerical experiments show that the spectrum of a linearized QCF
  operator is identical to the spectrum of a linearized energy-based
  quasi-nonlocal quasicontinuum operator (QNL), which we know from our
  previous analyses to be positive below the critical load. However,
  the QCF operator is non-normal and it turns out that it is not
  generally positive definite, even when all of its eigenvalues are
  positive. Using a combination of rigorous analysis and numerical
  experiments, we investigate in detail for which choices of
  ``function spaces'' the QCF operator is stable, uniformly in the
  size of the atomistic system.

  Force-based multi-physics coupling methods are popular techniques to
  circumvent the difficulties faced in formulating consistent
  energy-based coupling approaches. Even though the QCF method is
  possibly the simplest coupling method of this kind, we anticipate
  that many of our observations apply more generally.
\end{abstract}

\maketitle


\section{Introduction}

Low energy equilibria for crystalline materials are typically
characterized by localized defects that interact with their
environment through long-range elastic fields.  Atomistic-to-continuum
coupling methods seek to make the accurate computation of such
problems possible by using the accuracy of atomistic modeling only in
the neighborhood of defects where the deformation is highly
non-uniform. At some distance from the defects, sufficient accuracy
can be obtained by the use of continuum models, which facilitate the
reduction of degrees of freedom. The accuracy of the atomistic model
at the defect combined with the efficiency of a continuum model for
the far field enables, in principle, the reliable simulation of
systems that are inaccessible to pure atomistic or pure continuum
models.

Typical test problems for atomistic-to-continuum coupling methods have
been dislocation formation under an indenter, crack tip deformation,
and deformation and fracture of grain
boundaries~\cite{Miller:2003a}. In each of these problems, the crystal
deforms quasistatically until the equilibrium equations become
singular, for example, when a dislocation is formed or moves or when a
crack tip advances.  Depending of the nature of the singularity, the
crystal will then typically undergo a dynamic process when further
loaded.

The quasicontinuum (QC) approximation models the continuum region by
using an energy density that exactly reproduces the lattice-based
energy density at uniform strain (the Cauchy-Born rule).  Several
variants of the QC approximation have been proposed that differ in how
the atomistic and continuum regions are coupled~\cite{Ortiz:1995a,
  Shenoy:1999a, Miller:2003a, Shimokawa:2004, E:2006, Dobson:2008a,
  Gunzburger:2008a, Gunzburger:2008b,
  blan05}. Analyses of QC approximation have been given in
  \cite{E:2005a,Ortner:2008a,mingyang,LinP:2003a, LinP:2006a,Dobson:2008b}.
   We refer to \cite{doblusort:qce.stab} for a detailed review
of the formulation and analysis, relevant to the present work, of
different QC methods.  Other coupling models are analyzed
in~\cite{park07, Badia:2007, prud06}.

In \cite{doblusort:qce.stab}, we have begun to investigate whether the
QC method can reliably predict the formation of defects. The main
ingredient to establish whether or not this is the case is a sharp
analysis to predict under which conditions the QC method is
``stable.'' More precisely, we ask whether there exist ``stable''
solutions of the QC method {\em up to a critical load}. We have begun
to investigate this question in some depth for the most common
energy-based QC formulations in \cite{doblusort:qce.stab}. In the
present paper, we present a corresponding sharp stability analysis for
the force-based quasicontinuum (QCF) method~\cite{Dobson:2008a,
  curt03, Shenoy:1999a}.

We focus on a one dimensional periodic chain with next-nearest
neighbour pair interactions, which is introduced in Section
\ref{sec:intro:model_problem}. For this model, the uniform
configuration ceases to be stable when the applied tensile strain
reaches a critical value (fracture).

For the atomistic model and for energy-based QC formulations,
coercivity (positivity) of the second variation evaluated at the equilibrium
solution provides the natural notion of stability. However, the QCF
method, which we describe in Sections \ref{sec:model:qcf} and
\ref{sec:proj_forces}, leads to non-conservative equilibrium
equations, and therefore, positivity of the linearized QCF operator
may be an inappropriate notion of stability. Indeed, we prove in Section
\ref{sec:no_coercivity} that, generically, the linearized QCF operator
is indefinite.

As a consequence, we consider two further notions of stability. First,
we investigate for which choices of discrete function spaces (that is,
for which choices of topologies) does the linearized QCF operator have
an inverse that is bounded uniformly in the size of the atomistic
system. In Section \ref{sec:qcf_ana:opstab}, we present several sharp
stability results as well as interesting counterexamples. However,
these operator stability results do not necessarily correspond to any
physical notion of stability. Hence, in Section
\ref{sec:qcf_ana:dynstab}, we propose the notion of {\em dynamical
  stability}, which can be reduced to certain properties of the
eigenvalues. A careful numerical study suggests that the spectrum of
the linearized QCF operator and that of the linearized quasi-nonlocal
QC operator (QNL) (see \cite{Shimokawa:2004} and Section
\ref{sec:qnl}) are identical. Combined with our previous results
\cite{doblusort:qce.stab}, this indicates that the QCF method is {\em
  dynamically stable} up to the critical load for fracture.


\clearpage
\section{The force-based quasicontinuum method}

\subsection{The atomistic model problem}
\label{sec:intro:model_problem}
We consider deformations from the reference lattice $\eps \Z$, where
$\eps > 0$ is a scaling that we will fix below. For the sake of
simplicity, we admit only deformations which are periodic
displacements from the uniform state $y_F = F\eps\Z = ( F
\eps\ell)_{\ell \in \Z}$, that is, we admit deformations from the
space
\begin{align*}
  \Ys_F =~& y_F + \Us \qquad \text{where} \\
  \Us =~& \big\{ u \in \R^\Z : u_{\ell+2N} = u_\ell \text{ for }
  \ell \in \Z, \text{ and } {\textstyle \sum_{\ell = -N+1}^N}
  u_\ell = 0 \big\}.
\end{align*}
We call $F$ the {\em macroscopic deformation gradient}, and we set
$\eps = 1/N$ throughout. Although the energies and forces are defined
for general $2N$-periodic displacements, we only admit those with zero
mean, as is common for continuum problems with periodic boundary
conditions, in order to obtain unique solutions to the equilibrium
equations.

We consider only nearest-neighbor and next-nearest neighbor pair
interactions so that the {\em potential energy per period} of a
deformation $y \in \Ys_F$ is given by
\begin{equation*}
  \E_{\rm a}(y) = \eps \sum_{\ell = -N+1}^N \big(\phi(y_\ell')
  + \phi(y_\ell' + y_{\ell+1}')\big),
\end{equation*}
where
\begin{displaymath}
  y_\ell' = \eps^{-1}(y_\ell - y_{\ell-1}),
\end{displaymath}
and where $\phi$ is a Lennard-Jones type interaction potential:
\begin{enumerate}
  \ilist
\item $\phi \in \CC^3((0, +\infty); \R)$,
\item there exists $r_* > 0$ such that $\phi$ is convex in $(0, r_*)$
  and concave in $(r_*, +\infty)$.
\item $\phi^{(k)}(r) \rightarrow 0$ rapidly as $r \nearrow \infty$,
  for $k = 0, \dots, 3$.
\end{enumerate}
Assumption (iii) is not strictly necessary for our analysis but serves
to motivate that next-nearest neighbour interactions are typically
dominated by nearest-neighbour terms.

We assume that the atomistic system is subject to $2N$-periodic
external forces $(f_\ell)_{\ell \in \Z}$ with zero mean, i.e., $f \in
\Us$, so that the {\em total energy per period} takes the form
\begin{equation*}
  \E_a^{tot}(y) = \E_a(y) - \eps \sum_{\ell=-N+1}^N f_\ell y_\ell.
\end{equation*}

Equilibria $y \in \Ys_F$ of the atomistic total energy are solutions
to the equilibrium equations
\begin{equation}
  \label{eq:forces_atom}
  \Fs_{\rm a, \ell}(y)+f_\ell=0, \qquad -\infty<\ell<\infty,
\end{equation}
where the (scaled) atomistic forces $\Fs_{\rm a}: \Ys_F \to \Us^*$
are defined by
\begin{displaymath}
  \Fs_{\rm a, \ell}(y):=- \frac1\eps \frac{\pp \E_{\rm a}(y)}{\pp y_\ell},
  \qquad -\infty<\ell<\infty,
\end{displaymath}
and where $\Us^*$ is the space of linear functionals on $\Us.$
We remark that the translational invariance of the atomistic energy
implies that $\Fs_{\rm a, \ell}(y)$ has zero mean,
\begin{equation}
  \label{invatom}
  \sum_{\ell = -N+1}^N \Fs_{\rm a, \ell}(y)
  =\frac{d}{ds}\E_{\rm a}\left(y-\frac{s}{\eps}e\right)\Big|_{s = 0}
  = 0,
\end{equation}
where $e = (1)_{\ell \in \Z}$ is the unit translation vector. Thus, we
see that, at least heuristically, the external force vector $f$ lies
indeed in the range of the atomistic force operator.

We note, moreover, that $y_F$ is an equilibrium of the atomistic
energy, that is
\begin{displaymath}
  \Fs_{\rm a, \ell}(y_F) = 0 \qquad -\infty<\ell<\infty,
  \quad \text{for all } F > 0.
\end{displaymath}
The question which we will investigate in this paper, beginning in
Section \ref{sec:stab_yF}, is for which $F$ it is a {\em stable
  equilibrium}, and whether the force-based QC method is able to
predict the stability of $y_F$.

\subsection{The local quasicontinuum approximation}
We begin by observing that the atomistic energy can be rewritten as a
sum over the contributions from each atom,
\begin{equation}
  \label{eq:rewrite_Ea}
  \begin{split}
    \E_{\rm a}(y) =~& \eps \sum_{\ell = -N+1}^N E^{\rm a}_\ell(y)
    \qquad \text{where} \\
    E^{\rm a}_\ell(y) =~& \half\big[ \phi(y_\ell')
    + \phi(y_{\ell+1}')
    + \phi(y_{\ell-1}' + y_\ell')
    + \phi(y_{\ell+1}'+y_{\ell+2}') \big].
  \end{split}
\end{equation}
If $y$ is ``smooth'', that is, if $y_\ell'$ varies slowly, then the
atomistic energy can be accurately approximated by the Cauchy--Born or
{\em local quasicontinuum} energy
\begin{equation}
  \label{eq:defn_qcl}
  \begin{split}
    \E_{\rm qcl}(y) =~& \eps \sum_{\ell = -N+1}^N E_\ell^{\rm c}(y),
    \qquad \text{where} \\
    E_\ell^{\rm c}(y) =~& \half \big[ \phi(y_\ell') + \phi(y_{\ell+1}')
    + \phi(2y_\ell') + \phi(2y_{\ell+1}') \big]
    = \half \big[ \phi_{\rm cb}(y_\ell') + \phi_{\rm cb}(y_{\ell+1}') \big],
  \end{split}
\end{equation}
where $\phi_{\rm cb}(r) = \phi(r) + \phi(2r)$ is the {\em
  Cauchy--Born} stored energy density.

In this approximation we have replaced the next-nearest neighbor
interactions by nearest neighbor interactions to obtain a model with
{\em stronger locality}. This makes it possible to coarsen the model
(to remove degrees of freedom), which eventually leads to significant
gains in efficiency~\cite{Miller:2003a, Dobson:2008a}. However, in the
present work we will not consider this additional step.

An equilibrium $y\in\Ys_F$ of the local QC energy is a solution to the
equilibrium equations
\begin{equation}
  \label{eq:forces_qcl}
  \Fs_{\rm c, \ell}(y) + f_\ell = 0, \qquad -\infty<\ell<\infty,
\end{equation}
where the (scaled) local QC forces $\Fs_{\rm c}: \Ys_F \to \Us^*$ are
defined by
\begin{displaymath}
  \Fs_{\rm c, \ell}(y) := - \frac{1}{\eps} \frac{\pp \E_{\rm a}(y)}{\pp y_\ell},
  \qquad -\infty<\ell<\infty.
\end{displaymath}
As in \eqref{invatom} it follows that the vector $\Fs_{\rm c}(y)$ has
zero mean.

\subsection{The force-based quasicontinuum approximation}
\label{sec:model:qcf}

If a deformation $y$ is ``smooth'' except in a small region of the
domain, then it is desirable to couple the accurate atomistic
description with the efficient continuum description.  The force-based
quasicontinuum (QCF) approximation achieves this by mixing the
equilibrium equations of the atomistic model with those of the
continuum model without any interface or transition region.

Suppose that $y$ is ``smooth'' except in a region $\As := \{ -K,
\dots, K \}$, where $K > 1.$ We call $\As$ the {\em atomistic region}
and $\Cs = \{-N+1,\dots, N\} \setminus \As$ the {\em continuum
  region.}  The force-based QC approximation is obtained by evaluating
the forces in the atomistic region by the full atomistic model
\eqref{eq:forces_atom} and the forces in the continuum region by the
local QC model~\eqref{eq:forces_qcl}. This yields the QCF operator for
the (scaled) forces $\Fs_{\rm qcf} : \Ys_F \to \Us^*$, defined by
\begin{equation}
  \label{eq:defn_Fqc_1}
  \Fs_{{\rm qcf}, \ell}(y) := \cases{
    \Fs_{\rm a, \ell}(y), & \quad \text{if~} \ell \in \As, \\
    \Fs_{\rm c, \ell}(y), & \quad \text{if~} \ell \in \Cs.
  }
\end{equation}

Force-based coupling methods such as \eqref{eq:defn_Fqc_1} are
trivially consistent (provided the continuum model is consistent with
the atomistic model) and are therefore a natural remedy for the
inconsistencies one observes when formulating simple energy-based
coupling methods such as the original QC method \cite{Ortiz:1995a}.
Similar constructions have appeared in the literature under several
different names and for various applications (e.g., FeAt
\cite{kohlhoff}, CADD \cite{cadd}, or {\em brutal force mixing}
\cite{hybrid_review}). In the context of the QC method this method was
first described in \cite{Dobson:2008a}, where it was shown that the
force-based QC method is the limit of the so-called {\em ghost-force
  correction iteration} \cite{Shenoy:1999a}.  A basin of attraction
and rate for the convergence the ghost-force correction iteration to
the force-based QC method was given in \cite{Dobson:2008a}.  Sharp
stability estimates for the ghost-force correction iteration are given
in \cite{qcf.iterative}.

Unfortunately, the forces generated by the QCF method are
non-conservative, and hence cannot be associated with an
energy. Moreover, even though both the atomistic forces $\Fs_{\rm
  a}(y)$ and the local QC forces $\Fs_{\rm c}(y)$ have zero mean, it
turns out that this is false for the mixed forces $\Fs_{\rm
  qcf}(y)$. A straightforward computation shows that
\begin{align*}
  \sum_{\ell = -N+1}^N \Fs_{{\rm qcf}, \ell}(y)
  =~& \eps^{-1} [ 2 \phi'(2y'_{-K})  - \phi'(y'_{-K}+y'_{-K-1})
  - \phi'(y'_{-K+1}+y'_{-K})]\\
  &- \eps^{-1} [ 2 \phi'(2 y'_{K+1})  - \phi'(y'_{K+2}+y'_{K+1})
  - \phi'(y'_{K+1}+y'_{K})],
\end{align*}
which is in general non-zero. After introducing the necessary
notation, we will overcome this difficulty by defining a variational
form of the QCF method, which effectively projects the QCF forces onto
the correct range.

\subsection{Norms and variational notation}
For future reference, we recall the backward first difference
$v_\ell'=\eps^{-1}(v_{\ell}-v_{\ell-1})$ and also define the centered
second difference $v_\ell'' = \eps^{-2}(v_{\ell+1} - 2 v_\ell +
v_{\ell-1}).$

For displacements $v \in \Us$ and $1 \leq p \leq \infty,$ we define
the $\ell^p_\eps$ norms,
\begin{displaymath}
  \lpnorm{v}{p} := \cases{
    \Big( \eps\sum_{\ell=-N+1}^N |v_\ell|^p \Big)^{1/p},
    &1 \leq p < \infty, \\
    \max_{\ell = -N+1,\dots,N} |v_\ell|,
    &p = \infty,
  }
\end{displaymath}
and let $\Us^{0,p}$ denote the space $\Us$ equipped with the
$\ell^p_\eps$ norm.  We further define the $\Us^{1,p}$ norm
\begin{displaymath}
  \|v\|_{\Us^{1,p}} := \|v'\|_{\ell^p_\eps},
\end{displaymath}
and let $\Us^{1,p}$ denote the space $\Us$ equipped with the
$\Us^{1,p}$ norm.  Similarly, we define the space $\Us^{2,p}$ and its
associated $\Us^{2,p}$ norm.

The inner product associated with the $\ell^2_\eps$ norm is
\begin{equation*}
  \<v,w\> := \eps \sum_{\ell=-N+1}^N v_\ell w_\ell \qquad
  \text{ for } v, w \in \Us.
\end{equation*}
We have defined the norms $\|\cdot\|_{\ell^p_\eps}$ and the inner
product $\<\cdot,\cdot\>$ on $\Us,$ though we will also apply them for
arbitrary vectors from $\R^{2N}$.

The external force $f = (f_\ell)_{\ell \in \Z}$ is a $2N$-periodic
mean zero vector, and we have seen that the atomistic forces and the
forces in the QCL method are also $2N$-periodic mean zero vectors.
Using the inner product, we can view $f_\ell$ as a linear functional
on $\Us.$ We recall that the space of linear functionals on $\Us$ is
denoted by $\Us^*,$ and we note that each such $T \in \Us^*$ has a
unique representation as a zero mean $2N$-periodic vector $g_T \in
\Us$,
\begin{equation}
  \label{eq:representation}
  T[v] = \< g_T, v \> \qquad \forall v \in \Us.
\end{equation}
We will normally not make a distinction between these
representations. For example, an external force vector $f$ may be
equally interpreted as a linear functional (i.e., $f \in \Us^*$), or
identified with its Riesz representation (i.e., $f \in \Us$).

For $g \in \Us^*$, $s = 0,1,$ and $1 \leq p \leq \infty$, we define
the negative norms $\| g \|_{\Us^{-s,p}}$ as follows:
\begin{equation}\label{meannorm}
  \|g\|_{\Us^{-s,p}}
  := \sup_{\substack{v \in \Us\\\|v\|_{\Us^{s,q}} = 1}} \< g, v \>,
\end{equation}
where $1 \leq q \leq \infty$ satisfies $\frac{1}{p} + \frac{1}{q} =
1.$ We let $\Us^{-s,p}$ denote the space $\Us^*$ equipped with the
$\Us^{-s,p}$ norm.

Since we can identify elements of $\Us^*$ with elements of $\Us$, we
can investigate the relationship between the $\Us^{-0,p}$ and
$\Us^{0,p}$-norms. This will be useful later on in our analysis. It
turns out that $\|\cdot\|_{\Us^{-0,p}} \neq \|\cdot\|_{\Us^{0,p}}$ in
general, but that the following equivalence relation holds:
\begin{equation}
  \label{eq:norm_equiv}
  \|u\|_{\Us^{-0,p}} \leq \|u\|_{\Us^{0,p}}
  \leq 2 \|u\|_{\Us^{-0,p}}
  \quad \text{for all } u \in \Us.
\end{equation}
To see this, we note that the inequality $\|u\|_{\Us^{-0,p}} \leq
\|u\|_{\Us^{0,p}}$ follows from \eqref{meannorm} and H\"older's
inequality.  To prove the second inequality, we use that fact that,
for $u \in \Us$,
\begin{displaymath}
\|u\|_{\Us^{0,p}}
  =\sup_{\substack{v \in \R^{2N}\\\|v\|_{\ell^q_\eps} = 1}} \< u, v \>
  =\sup_{\substack{ v \in \R^{2N}\\ \|v\|_{\ell^q_\eps} = 1}} \< u, v -\bar v\>,
\end{displaymath}
where $\bar v = \frac{1}{2N} \sum_{j=-N+1}^N v_j$ of
$v\in\R^{2N}$. Thus, we can estimate
\begin{displaymath}
  \|u\|_{\Us^{0,p}}
  \le  \|u\|_{\Us^{-0,p}} \sup_{\substack{ v \in \R^{2N}\\ \|v\|_{\ell^q_\eps} = 1}}
  \|v-\bar v\|_{\ell^q_\eps}
  \le 2\|u\|_{\Us^{-0,p}},
\end{displaymath}
where we also used the fact that, by H\"older's inequality, $\| \bar v
\|_{\ell^q_\eps} \le \|v\|_{\ell^q_\eps}$ for any $v\in\R^{2N}$.

\subsection{Projection of non-conservative forces}
\label{sec:proj_forces}
If we interpret forces as elements of $\Us^*,$ then it is natural to
consider the following variational formulation of the QCF method,
\begin{equation}
  \label{eq:qcf_variational}
  \big\< \Fs_{\rm qcf}(y) + f, u \big\> = 0 \qquad \forall\, u \in \Us.
\end{equation}
In other words, \eqref{eq:qcf_variational} requires that $\Fs_{\rm
  qcf}(y) + f = 0$ as a functional in $\Us^*$. This formulation
guarantees that the QCF operator has the correct range.

To obtain an atom-based description of the equilibrium equations, we
explicitly compute the representation of $\Fs_{{\rm qcf}}(y) \in
\Us^*$ as an element of $\Us$ (see also \eqref{eq:representation}),
that is as a zero mean $2N$-periodic vector
$\Ps_\Us \Fs_{{\rm qcf}}(y)$, where $\Ps_\Us$ is defined by
\begin{equation*}
  \big(\Ps_\Us v\big)_\ell = v_\ell
  - \frac{1}{2N} \sum_{j=-N+1}^N v_j.
\end{equation*}
With this notation, the variational equilibrium equations can be
understood as {\em projected} equilibrium equations in atom-based
form,
\begin{equation}
  \label{eq:qcf_projected}
  \big(\Ps_\Us \Fs_{{\rm qcf}}(y) \big)_\ell + f_\ell = 0,
  \qquad -\infty<\ell<\infty.
\end{equation}
The equivalent formulations \eqref{eq:qcf_variational} and
\eqref{eq:qcf_projected} define the correct force-based QC method for
the periodic model problem defined in Section
\ref{sec:intro:model_problem}.

\begin{remark}
  The projection of the QCF equilibrium system is an artifact of the
  periodic boundary conditions. For the displacement boundary
  conditions that we analyzed in \cite{dobs-qcf2}, or for the mixed
  boundary conditions that are considered in \cite{dobsonluskin08}, this
  projection is not necessary.
\end{remark}

\section{Stability of a uniform deformation}
\label{sec:stab_yF}
It is easy to see that, in the absence of external forces, the
uniformly deformed lattice $y = y_F$ is an equilibrium of the
atomistic energy as well as the local QC energy, that is
\begin{displaymath}
  \Fs_{\rm a}(y_F) = 0 \quad \text{and} \quad
  \Fs_{\rm c}(y_F) = 0 \qquad \text{for all }  F > 0.
\end{displaymath}
For some values of $F,$ the equilibrium will be stable, by which we mean that
the second variation
\begin{equation*}
  \E_a''(y_F)[u,v] = \eps \sum_{\ell = -N+1}^N \big\{ \phi_F'' u_\ell' v_\ell'
  + \phi_{2F}'' (u_\ell' + u_{\ell+1}')(v_\ell'+v_{\ell+1}') \big\}, \qquad
  \text{for } u\in\Us,
\end{equation*}
where
\begin{displaymath}
  \phi_F'' := \phi''(F) \quad \text{and} \quad
  \phi_{2F}'' := \phi''(2F),
\end{displaymath}
is positive definite, that is,
\begin{displaymath}
  \E_{\rm a}''(y_F)[u, u] > 0 \qquad \forall u \in \Us \setminus \{0\}.
\end{displaymath}
(We note that a second variation, e.g. $\E_{\rm a}''(y_F)$, may be understood
either as a bilinear form on $\Us$ or a linear operator from $\Us$ to
$\Us^*$.  It can also be expressed as a Hessian matrix with
respect to a given basis for the vector space $\Us.$)

In order to avoid having to distinguish several cases, we will assume
throughout our analysis that $F \geq r_*/2,$ which implies by property
(ii) of the interaction potential that $\phi_{2F}'' \leq 0.$
This assumption holds for most realistic interaction potentials so
long as the chain is not under extreme compression.

As above, we can evaluate the second variation of the local QC energy at $y =
y_F$,
\begin{displaymath}
  \E_{\rm qcl}''(y_F)[u,v] = \eps \sum_{\ell = -N+1}^N A_F u_\ell' v_\ell',
\end{displaymath}
where $A_F$ is the {\em elastic modulus} of the continuum model,
\begin{displaymath}
  A_F := \phi_{\rm cb}''(F) = \phi''_F + 4 \phi''_{2F}.
\end{displaymath}
Thus, we say that $y_F$ is stable for the local QC approximation if
$\E_{\rm qcl}''(y_F)[u,u] > 0$ for all $u \in \Us \setminus \{0\}$.

In \cite{doblusort:qce.stab}, we have given explicit characterizations
for which $F$ the equilibrium $y_F$ is stable in the atomistic model
and in several energy-based QC models. The results for the atomistic
and the local QC models are summarized in the following proposition.

\begin{proposition}[cf. Prop. 1 and 2 in
  \cite{doblusort:qce.stab}]
  \label{th:stab_a_c}
  Let $F \geq r_* / 2$ then the second variations $\E_{\rm a}''(y_F)$,
  respectively $\E_{\rm qcl}''(y_F)$, are positive definite if and
  only if
  \begin{displaymath}
    A_F - \lambda_N^2 \eps^2 \phi''_{2F} > 0, \quad
    \text{respectively if} \quad  A_F > 0,
  \end{displaymath}
  where $2 \leq \lambda_N \leq \pi$.
\end{proposition}

If we denote the critical strains which divide the regions of
stability for the atomistic and QCL models, respectively, by $F_{\rm
  a}^*$ and $F_{\rm c}^*$, then a relatively straightforward error
analysis \cite[Sec. 5]{doblusort:qce.stab} shows that $F_{\rm a}^* =
F_{\rm c}^* + O(\eps^2)$, that is, the QCL model accurately reproduces
the onset of a fracture instability. In the following section, we
investigate whether or not the QCF method has a similar property.

\section{Sharp Stability of the Force-based QC Method}
\label{sec:qcf}
A trivial consequence of the definition of $\Fs_{\rm qcf}$ in
\eqref{eq:defn_Fqc_1} is that $y = y_F$ is also a solution of the QCF
equilibrium equations \eqref{eq:qcf_projected},
\begin{displaymath}
  \Fs_{\rm qcf}(y_F) = 0 \qquad \text{for all } F > 0.
\end{displaymath}
(As a matter of fact, this means that the QCF method is consistent;
though this is not the focus of the present work.)

To investigate the stability of the QCF method we define the
linearized QCF operator $L_{{\rm qcf}, F}:= - \Fs'_{{\rm
    qcf}}(y_F): \Us \to \Us^*$ by
\begin{equation*}
  \< L_{{\rm qcf}, F} u, v \> := - \<\Fs'_{{\rm qcf}}(y_F)[u], v\>
  \qquad \text{ for all } u,\,v \in \Us.
\end{equation*}
The equilibrium equations for the linearized force-based approximation
are then given by $u\in \Us$ satisfying
\begin{equation*}
  \< L_{{\rm qcf}, F} u, v \> = \< f, v\> \qquad \text{ for all } v \in \Us,
\end{equation*}
or in functional form
\begin{displaymath}
  \Ps_\Us L_{{\rm qcf}, F} u = f.
\end{displaymath}
We remark that, while $L_{{\rm qcf}, F} \in L(\Us,\, \Us^*)$, the
projected operator $\Ps_\Us L_{{\rm qcf}, F}$ may be interpreted as a
map from $\Us$ to $\Us$.

\subsection{Lack of coercivity}
\label{sec:no_coercivity}
Since the force field $\mathcal{F}_{{\rm qcf}}(y)$ is non-conservative
and the linearized QCF operator $L_{{\rm qcf},F}$ is not the second
variation of an energy functional, positivity (or coercivity) of
$L_{{\rm qcf},F}$ may be the incorrect notion of stability for the QCF
model.  Indeed, it turns out that, if $N$ is large, then
$L_{{\rm qcf},F}$ cannot be positive definite.

\begin{theorem}
  \label{th:qcf_ana:notcoerc}
  Let $\phi_F'' > 0$ and $\phi_{2F}'' \neq 0$, then there exist
  constants $C_1, C_2$ which may depend on $\phi_F''$ and
  $\phi_{2F}''$, such that, for $N$ sufficiently large and for $2 \leq
  K \leq N/2$,
  \begin{displaymath}
    - C_1 N^{1/2} \leq
    \inf_{\substack{u \in \Us \\ \|u'\|_{\ell^2_\eps} = 1}}  \< L_{{\rm qcf}, F} u, u \>
    \leq - C_2 N^{1/2}.
  \end{displaymath}
\end{theorem}

In \cite{dobs-qcf2}, we have shown this result for a Dirichlet
boundary value problem.  The proof carries over from the Dirichlet
case almost verbatim and is therefore omitted. As a matter of fact,
the test function which we explicitly constructed in the proof of
Lemma 4.1 in \cite{dobs-qcf2} is already periodic and, after shifting
it to have zero mean, can therefore be used again to prove Theorem
\ref{th:qcf_ana:notcoerc}.

Theorem \ref{th:qcf_ana:notcoerc} forces us to consider alternative
notions of stability. For example, one could understand $L_{{\rm qcf}, F}$
as a linear operator between appropriately chosen discrete
function spaces, determine for which values of $F$
it is bijective, and estimate the norm of its inverse.
Physically, this measures the magnitude of the response of the equilibrium
configuration to perturbations in external forces, and
in Section~\ref{sec:qcf_ana:opstab} we attempt to find
the largest interval surrounding $F=1$ and consider this region to be
the approximation of the stable region given by operator stability of $L_{{\rm qcf}, F}.$
However, an operator can be bijective and its inverse can have bounded norm even when it
has negative eigenvalues, so for a general equilibrium state such as $y_F,$
invertibility alone is not a suitable criterion for determining
 the ``physical'' stability.
To be able to decide whether a stable
equilibrium of the QCF equations is also stable in a physical
sense, we propose a notion of dynamical stability in Section~\ref{sec:qcf_ana:dynstab}.

\subsection{Stability as a linear operator}
\label{sec:qcf_ana:opstab}
Since $\Us$ is a finite-dimensional linear space, the choice of
topology with which we equip it is unimportant to the question whether
$L_{{\rm qcf}, F}$ is invertible. However, it has surprising
repercussions when we analyze an operator norm of the inverse, that is
$\| L_{{\rm qcf}, F}^{-1} \|$, in the limit as $N \to \infty$.

Our strongest and simplest result is obtained when we view $\Ps_\Us
L_{{\rm qcf}, F}$ as a map from $\Us^{2,\infty}$ to $\Us^{0,\infty}.$

\begin{theorem}
  \label{th:qcf_ana:sharp_stab_W2inf}
  If $|\phi_F''| - (4 + 2\eps) |\phi_{2F}''| > 0,$ then $\Ps_\Us
  L_{{\rm qcf},F} : \Us \to \Us$ is bijective and
  \begin{displaymath}
    \big\| (\Ps_\Us L_{{\rm qcf},F})^{-1} \big\|_{L(\Us^{0, \infty},\ \Us^{2,\infty})}
    \leq \frac{1}{|\phi_F''| - (4+2\eps) |\phi_{2F}''|}.
  \end{displaymath}
\end{theorem}
\begin{proof}
  Recalling the definition of $L_{{\rm qcf},F}$, we can rewrite this
  operator in the form
  \begin{displaymath}
    \mathcal{P}_\Us L_{{\rm qcf},F} = \phi_F'' L_1
    + \phi_{2F}'' \mathcal{P}_\Us \tilde L_2,
  \end{displaymath}
  where $L_1$ and $\tilde L_2$ are given by
  \begin{align*}
    (L_1 u)_\ell =~& - \eps^{-2} (u_{\ell+1} - 2 u_\ell + u_{\ell-1}),
    \quad \text{and} \\
    (\tilde L_2 u)_\ell =~& \cases{
      - \eps^{-2} (u_{\ell+2} - 2 u_\ell + u_{\ell-2}), & \ell = -K, \dots, K, \\
      -4\eps^{-2} (u_{\ell+1} - 2 u_\ell + u_{\ell-1}), &
      \text{otherwise.}  }
  \end{align*}
  We note that $\mathcal{P}_\Us L_1 = L_1$ which is why we have
  included the projection only in the second-neighbor operator.

  The projection of $\tilde L_2$ given by $\mathcal{P}_\Us \tilde L_2$
  is
  \begin{displaymath}
    (\mathcal{P}_\Us \tilde L_2 u)_\ell
    = (\tilde L_2 u)_\ell - \frac{\eps}{2} \sum_{j = -N+1}^N (\tilde L_2 u)_j.
  \end{displaymath}
  We will prove below that
  \begin{equation}
    \label{eq:qcf_ana:bndL2tilde}
    \| \mathcal{P}_\Us \tilde L_2 \|_{L(\Us^{2,\infty},\ \Us^{0,\infty})}
    \leq 4 + 2 \eps.
  \end{equation}
  Assuming that this bound is established, we obtain
  \begin{align*}
    \| \Ps_\Us L_{{\rm qcf,F}} u \|_{\ell^\infty} \geq~&
    |\phi_F''| \| L_1 u \|_{\ell^\infty}
    - |\phi_{2F}''| \| \mathcal{P}_\Us \tilde L_2 u \|_{\ell^\infty} \\
    \geq~& (|\phi_F''| - (4+2\eps) |\phi_{2F}''|) \| u''\|_{\ell^\infty},
  \end{align*}
  which is equivalent to the statement of the theorem.

  To prove \eqref{eq:qcf_ana:bndL2tilde}, we note that, for $\ell =
  -K, \dots, K$, we have
  \begin{displaymath}
    (\tilde L_2 u)_\ell = - (u_{\ell+1}'' + 2 u_{\ell}'' + u_{\ell-1}'')
    = - 4 u_{\ell}'' - (u_{\ell+1}'' - 2 u_{\ell}'' + u_{\ell-1}'').
  \end{displaymath}
  Using the first representation of $(\tilde L_2 u)_\ell$ above, we
  immediately see that (for $\ell$ from the continuum region this
  statement is trivial)
  \begin{displaymath}
    \big|(\tilde L_2 u)_\ell\big| \leq 4 \| u''\|_{\ell^\infty}
    \qquad \text{for } \ell = -N+1, \dots, N.
  \end{displaymath}
  From the second representation of $(\tilde L_2 u)_\ell$, we obtain
  \begin{align*}
    \sum_{\ell = -N+1}^N  (\tilde L_2 u)_\ell
    =~& - 4\sum_{\ell = -N+1}^N u_\ell''
    - \sum_{\ell= -K}^K (u_{\ell+1}'' - 2 u_{\ell}'' + u_{\ell-1}'') \\
    =~& - u_{K+1}'' + u_{K}'' + u_{-K}'' - u_{-K-1}'',
  \end{align*}
  and hence,
  \begin{align*}
    \big|(\mathcal{P}_\Us \tilde L_2 u)_\ell\big|
    \leq~& \big|(\tilde L_2 u)_\ell\big|
    + \Bigg|\frac{\eps}{2} \sum_{j = -N+1}^N (\tilde L_2 u)_j\Bigg| \\
    \leq~& 4 \|u''\|_{\ell^\infty} + \smfrac{\eps}{2}
    \big(  |u_{K+1}''| + |u_{K}''| + |u_{-K-1}''| + |u_{-K}''| \big) \\[1mm]
    \leq~& (4 + 2\eps) \|u''\|_{\ell^\infty}.
  \end{align*}
  This establishes \eqref{eq:qcf_ana:bndL2tilde} and thus concludes
  the proof of the theorem.
\end{proof}

\begin{remark}
 With a small modification, Theorem
  \ref{th:qcf_ana:sharp_stab_W2inf} remains true for an arbitrary
  choice of the atomistic region $\As$. The correction $2\eps$ then
  needs to be replaced by $n_i \eps$ where $n_i$ is the number of
  interfaces between the atomistic and the continuum region.
\end{remark}

\begin{remark}
  Theorem \ref{th:qcf_ana:sharp_stab_W2inf} also holds in the case
  of the artificial Dirichlet boundary conditions analyzed in
  \cite{dobs-qcf2}. In that case, the projection $\mathcal{P}_\Us$ is
  not required and therefore the correction $2\eps$ does not occur at
  all.
\end{remark}

\medskip Theorem \ref{th:qcf_ana:sharp_stab_W2inf} is, in many
respects, a very satisfactory result. It shows that, except for a
small error, QCF is stable whenever the atomistic model is. However,
the choice of function space $\Us^{2,\infty}$ is somewhat unusual, and
it is highly unlikely that such a result would remain true in higher
dimensions, as it requires a regularity that is not normally exhibited
by linear elliptic systems.

It is therefore also interesting to analyze the QCF operator as a map
from $\Us^{1,p}$ to $\Us^{-1,p} = (\Us^{1,q})^*$, where $1 \leq p \leq
\infty$. However, we saw in \cite[Theorem 7.1]{dobs-qcf2} for a
Dirichlet problem that, for $1 \leq p < \infty$, the stability of
$L_{{\rm qcf},F}$ is not uniform in $N$. The following theorem, whose proof
is contained in Appendix~\ref{p}, establishes the same result for the
periodic model we consider in the present paper.
\begin{theorem}
  \label{th:qcf:infsup}
  Suppose that $\phi_F''>0$, $\phi_{2F}'' \in \R \setminus \{0\}$, and
  $1 \leq p < \infty$. Then there exists a constant $C > 0$, depending
  on $\phi_F''$ and $\phi_{2F}''$, such that, for $2 \leq K < N-2$,
  \begin{equation*}
    \| L_{{\rm qcf},F}^{-1} \|_{L(\Us^{-1,p},\ \Us^{1,p})} \geq C N^{1/p}.
  \end{equation*}
\end{theorem}

It remains to investigate the case $p = \infty$. The following result
is an extension of \cite[Thm. 5.1]{dobs-qcf2} to periodic boundary
conditions. Its proof is contained in Appendix~\ref{sec:qcf_stab}.

\begin{theorem}
  \label{th:qcf:stability}
  If $F \geq r_*/2$ and $\phi_F'' + 8 \phi_{2F}'' > 0,$ then
  \begin{displaymath}
   \| L_{{\rm qcf},F}^{-1} \|_{L(\Us^{-1,\infty},\ \Us^{1,\infty})}
   \leq \frac{2}{\phi_F'' + 8 \phi_{2F}''}.
  \end{displaymath}
\end{theorem}

Theorem \ref{th:qcf:stability} establishes {\em operator stability} of
the $L_{{\rm qcf},F}$ operator, uniformly in $N$, provided that
$\phi_F'' + 8 \phi_{2F}'' > 0$. Compared with with Proposition
\ref{th:stab_a_c} this result predicts a significantly smaller
stability region than either the atomistic model or the continuum
model. We employ numerical experiments to see whether the condition
$\phi_F'' + 8 \phi_{2F}'' > 0$ is sharp.

The norm $\| L_{{\rm qcf},F}^{-1} \|_{L(\Us^{-1,\infty},\
  \Us^{1,\infty})}$ is difficult to calculate explicitly, so we will
estimate it in terms of the $\ell^\infty$-operator norm of a related
matrix. To that end, we note that, according to Lemma
\ref{th:qcf_rep}, $L_{{\rm qcf}, F}$ can be represented in terms of a
conjugate operator, $E_{{\rm qcf}, F}$, by
\begin{displaymath}
  \< L_{{\rm qcf}, F} u, v \> = \< E_{{\rm qcf}, F} u', v' \>
  \qquad \forall u, v \in \Us.
\end{displaymath}
An explicit representation of $E_{{\rm qcf}, F}$ is provided in
\eqref{qcfvar2}. Formula \eqref{qcfvar2} gives
an $\R^{2N \times 2N}$ matrix representation for
$E_{{\rm qcf}, F}$ such that $E_{{\rm qcf}, F}e= A_F e$ where
$e = (1, \dots, 1)^T.$
It thus follows that the projected operator $\Ps_\Us E_{{\rm
    qcf}, F} : \R^{2N} \to \R^{2N}$ satisfies
\begin{equation*}
  \Ps_\Us E_{{\rm qcf}, F} : \Us \to \Us,
  \quad \text{and} \quad
  \Ps_\Us E_{{\rm qcf}, F} e = 0.
\end{equation*}
Here, and for the remainder of the section, we identify $\Us$ with the
subspace of $\R^{2N}$ of zero mean vectors. After these preliminary
remarks, we establish the following result.

\begin{proposition}
  \label{th:estimate_Lqcfinv}
  The QCF operator $L_{{\rm qcf}, F} : \Us \to \Us^*$ is invertible if
  and only if $(\Ps_\Us E_{{\rm qcf}, F} + e \otimes e) \in \R^{2N
    \times 2N}$ is invertible, and
  \begin{displaymath}
    \smfrac12 \| T \|_{\infty}
    \leq \| L_{{\rm qcf}, F}^{-1} \|_{L(\Us^{-1, \infty},\ \Us^{1, \infty})}
    \leq 2\| T \|_{\infty},
  \end{displaymath}
  where
  \begin{displaymath}
    T = \Ps_\Us (\Ps_\Us E_{{\rm qcf}, F} + e \otimes e)^{-1}
    \Ps_\Us,
  \end{displaymath}
  and where $\| T \|_\infty$ denotes the $\ell^\infty$-operator norm
  of $T$.
\end{proposition}
\begin{proof}
  The first statement follows from the discussion above.

  To prove the upper and lower bounds for $\| L_{{\rm qcf}, F}^{-1}
  \|_{L(\Us^{-1, \infty}, \Us^{1, \infty})},$ we first note that, by
  definition of $T$, it follows that
  \begin{displaymath}
    T \Ps_\Us E_{{\rm qcf}, F} f = \Ps_\Us E_{{\rm qcf}, F}  T f = f
    \qquad \text{for all } f \in \Us,
  \end{displaymath}
  that is, $T = (\Ps_\Us E_{{\rm qcf}, F})^{-1}$ on $\Us$. In
  addition, we also have $T e = 0$.

  Next, we note that
  \begin{displaymath}
    \begin{split}
      \frac{1}{\| L_{{\rm qcf},F}^{-1} \|_{L(\Us^{-1,\infty},\ \Us^{1,\infty})}}
      &=
      \inf_{\substack{\vb\in\Us \\ \lpnorm{\vb'}{\infty} = 1}}
      \sup_{\substack{\wb\in\Us \\ \lpnorm{\wb'}{1} = 1}}
      \<  L_{{\rm qcf},F} \vb,\,\wb \> \\
      &=
      \inf_{\substack{\vb\in\Us \\ \lpnorm{\vb'}{\infty} = 1}}
      \sup_{\substack{\wb\in\Us \\ \lpnorm{\wb'}{1} = 1}}
      \<  E_{{\rm qcf},F} \vb',\,\wb' \>
      =
      \frac{1}{\| E_{{\rm qcf},F}^{-1} \|_{L(\Us^{-0,\infty},\ \Us^{0,\infty})}}.
    \end{split}
  \end{displaymath}
  Since $T = (\Ps_\Us  E_{{\rm qcf},F})^{-1}$ on $\Us$, it follows that
  \begin{displaymath}
    \| L_{{\rm qcf},F}^{-1} \|_{L(\Us^{-1,\infty},\ \Us^{1,\infty})}
    = \| T \|_{L(\Us^{-0,\infty},\ \Us^{0,\infty})}.
  \end{displaymath}

  To prove the upper bound, we use \eqref{eq:norm_equiv} to estimate
  \begin{displaymath}
   \| T \|_{L(\Us^{-0,\infty},\ \Us^{0,\infty})}
   = \sup_{\substack{f \in \Us\\f \neq 0}}\frac{ \| T f\|_{\Us^{0,\infty}}}{
     \|f\|_{\Us^{-0,\infty}} }
   \leq \sup_{\substack{f \in \Us\\f \neq 0}}\frac{ \| T f\|_{\ell^\infty_\eps}}{
     \smfrac12 \|f\|_{\ell^\infty_\eps} }
   \leq 2\sup_{\substack{f \in \R^{2N}\\ f \neq 0}}\frac{ \| T f\|_{\ell^\infty_\eps}}{
     \|f\|_{\ell^\infty_\eps} }
   = 2\|T\|_\infty.
  \end{displaymath}
  To prove the lower bound, we first note that $T\Ps_\Us =T .$ We will
  also use the fact that $\| \Ps_\Us f \|_{\ell^\infty_\eps} \leq 2
  \|f\|_{\ell^\infty_\eps}$ for all $f \in \R^{2N}.$ Employing also
  \eqref{eq:norm_equiv} again, we can deduce that
  \begin{align*}
    \| T \|_{L(\Us^{-0,\infty},\ \Us^{0,\infty})}
   & =\sup_{\substack{f \in \R^{2N} \\ \Ps_\Us f \neq 0}}\frac{ \| T\Ps_\Us f\|_{\Us^{0,\infty}}}{
     \|\Ps_\Us f\|_{\Us^{-0,\infty}} }
    \ge  \sup_{\substack{f \in \R^{2N} \\ \Ps_\Us f\neq 0}}\frac{ \| T\Ps_\Us f\|_{\Us^{0,\infty}}}{
     \|\Ps_\Us f\|_{\Us^{0,\infty}} }
     =\sup_{\substack{f \in \R^{2N} \\ \Ps_\Us f\neq 0}}  \frac{ \| T f\|_{\ell^\infty_\eps}}{
     \|\Ps_\Us f\|_{\ell^\infty_\eps} }\\
   &
   \geq
   \sup_{\substack{f \in \R^{2N} \\ \Ps_\Us f\neq 0}}  \frac{ \| T f\|_{\ell^\infty_\eps}}{
     2 \| f\|_{\ell^\infty_\eps} }
    = \frac 12 \sup_{\substack{f \in \R^{2N} \\ f\neq 0}}  \frac{ \| T f\|_{\ell^\infty_\eps}}{
     \| f\|_{\ell^\infty_\eps} }
   = \smfrac12 \| T \|_\infty.
  \end{align*}
  The penultimate equality holds because $\Ps_\Us f = 0$ implies $T f
  = 0.$
\end{proof}

In Proposition \ref{th:estimate_Lqcfinv} we have reduced the
estimation of the operator norm of $L_{{\rm qcf}, F}^{-1}$ to the
computation of the $\ell^\infty$-operator norm (which is simply the
largest row sum) of a matrix $T\in \R^{2N \times 2N},$ which is
explicitly available (note that $\Ps_\Us=I-\frac 12 e\otimes e\in
\R^{2N \times 2N} $).

\begin{figure}
  \begin{center}
\includegraphics[width=10cm]{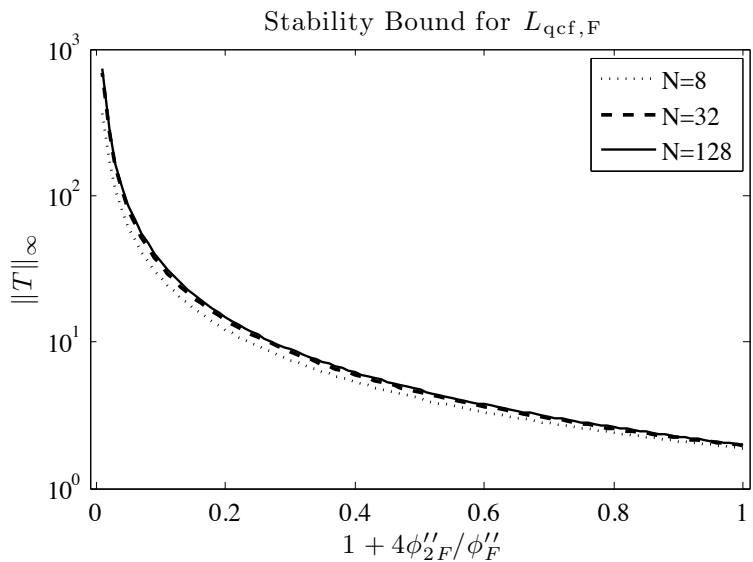} \\
    \caption{\label{fig:U1inf_stab} Computation of $\| T \|_{\infty}$
      where $T = \Ps_\Us (\Ps_\Us E_{{\rm qcf}, F} + e \otimes e)^{-1}
      \Ps_\Us$, which gives lower and upper bounds for $\| L_{{\rm
          qcf}, F}^{-1} \|_{L(\Us^{-1, \infty},\ \Us^{1, \infty})}$
      (cf.  Proposition~\ref{th:estimate_Lqcfinv}).  The graphs
      indicate that $L_{{\rm qcf},F}$ is stable as an operator from
      $\Us^{1,\infty}$ to $\Us^{-1,\infty}$, uniformly in $N$, for all
      macroscopic strains $F$ up to the critical strain for QCL and
      QNL.}
  \end{center}
\end{figure}

In Figure~\ref{fig:U1inf_stab}, we plot the norm of $T$ as a function
of $A_F / \phi_F'' = 1 + 4 \phi_{2F}'' / \phi_F''$. We clearly observe that
$L_{{\rm qcf},F}$ is in fact stable for all macroscopic gradients $F$ for
which $A_F>0$, that is, the bound required in Theorem
\ref{th:qcf:stability} is not sharp. Moreover, the numerical
experiments shown in Figure \ref{fig:U1inf_stab} support the following
conjecture.

\begin{conjecture}
  \label{th:qcf:stability:num}
  If $\phi_F'' + 4 \phi_{2F}'' > 0,$ then
  \begin{equation*}
   \| L_{{\rm qcf},F}^{-1} \|_{L(\Us^{-1,\infty},\ \Us^{1,\infty})}
   \leq \smfrac{1}{\phi''_F} \eta\Big(1 + 4 \smfrac{\phi''_{2F}}{
     \phi''_F}\Big),
  \end{equation*}
  where $\eta$ does not depend on $N$ or $K,$ but $\eta\Big(1 + 4
  \smfrac{\phi''_{2F}}{\phi''_F}\Big) \rightarrow \infty$ as $1 + 4
  \smfrac{\phi''_{2F}}{\phi''_F} \rightarrow 0.$
\end{conjecture}

In fact, the numerical experiments suggest that $\| L_{{\rm
    qcf},F}^{-1} \|_{L(\Us^{-1,\infty},\ \Us^{1,\infty})}$ grows
faster than $\frac{1}{\phi''_F + 4 \phi''_{2F}},$ which would imply
that an estimate such as the one in Theorem~\ref{th:qcf:stability},
but with the constant $8$ replaced by $4$ would be false.

\subsection{The quasi-nonlocal coupling method}
\label{sec:qnl}
In preparation for the following section, where we introduce another
notion of stability for the QCF method, we review a popular
energy-based coupling method. In the next section, we will make
numerical comparisons between this method and the QCF method.

The quasi-nonlocal quasicontinuum approximation (QNL)
\cite{Shimokawa:2004} was derived as a modification of the
energy-based QC approximation~\cite{Ortiz:1995a} in order to correct
the inconsistency at the atomistic-to-continuum
interface~\cite{Shenoy:1999a, Dobson:2008a}. In the case of
next-nearest neighbour pair interaction, the QNL method can be
formulated as follows. Nearest neighbor interaction terms are left
unchanged. A next-nearest neighbor interaction term
$\phi(\eps^{-1}(y_{\ell+1}-y_{\ell-1}))$ is left unchanged if atom
$\ell$ belongs to the atomistic region, but is replaced by a
Cauchy--Born approximation
\begin{displaymath}
  \phi(\eps^{-1}(y_{\ell+1}-y_{\ell-1})) \approx
  \half\big[ \phi(2y_\ell') + \phi(2y_{\ell+1}')],
  \quad \text{if } \ell \in \Cs.
\end{displaymath}
This process yields the QNL energy functional
\begin{displaymath}
  \begin{split}
    \E_{\rm qnl}(y) =\eps \sum_{\ell = -N+1}^N  \phi(y_\ell')
    + \eps \sum_{\ell \in \As} \phi(y_\ell' + y_{\ell+1}')
    + \eps \sum_{\ell \in \Cs} \half\big[ \phi(2y_\ell')
    + \phi(2y_{\ell+1}')\big].
  \end{split}
\end{displaymath}
We remark that the QNL method is consistent for our next-nearest
neighbour pair interaction model, and in particular, $y_F$ is an
equilibrium of the QNL energy functional in the absence of external
forces. Moreover, in \cite{doblusort:qce.stab} we have established the
following sharp stability result for the QNL method, which shows
that the QNL method is predictive up to the limit load for fracture.

\begin{proposition}[Proposition 3 in \cite{doblusort:qce.stab}]
  \label{th:stab_qnl}
  Suppose that $F \geq r_* / 2$ and that $K \leq N-1$, then $\E_{\rm
    qnl}''(y_F)$ is positive definite in $\Us$ if and only if $A_F >
  0$.
\end{proposition}

\subsection{Dynamical Stability}
\label{sec:qcf_ana:dynstab}
We have pointed out in Section \ref{sec:no_coercivity} that
operator stability for $L_{{\rm qcf},F}$ cannot guarantee that the
equilibrium $y_F$ is a stable equilibrium of the atomistic model
(e.g., a local minimum). To obtain at least a theoretical methodology
to determine stability of $y_F$ from the QCF operator alone, we
propose the notion of {\em dynamical stability}. The
dynamical system
\begin{equation*}
  \begin{split}
    \ddot u(t) &+ \Ps_\Us L_{{\rm qcf},F} u(t) = 0,\\
    u(0) &= u_0, \quad u'(0) = 0,
  \end{split}
\end{equation*}
has a unique solution $u \in \CC^\infty([0, +\infty);\, \Us)$. We call
this dynamical system {\em stable} if there exists a constant $C,$
independent of $N,$ such that
\begin{equation}
  \label{eq:qcf_ana:defn_dynstab}
  \| u(t) \|_{\ell^2_\eps} \leq C \| u_0 \|_{\ell^2_\eps} \qquad \forall t > 0,
\quad \forall u_0 \in \Us.
\end{equation}
This condition can be best understood in terms of the spectrum of
$\Ps_\Us L_{{\rm qcf},F}$. In numerical experiments, which are shown
in Table \ref{tbl:qcf_ana:conj_evals}, we have made the surprising
observation that $\Ps_\Us L_{{\rm qcf},F}$ and $\E''_{\rm qnl}(y_F)$
appear to have the same spectrum. This has led us to make the
following conjecture.

\begin{table}[t]
  \begin{equation*}
    \begin{array}{r|rrrrrrrr}
\textrm{N}&\phi_{2F}''= 0&\hbox{-}0.1&\hbox{-}0.15&\hbox{-}0.2&\hbox{-}0.25\\
\hline
 50&0
	&1.19e\hbox{-}010&9.93e\hbox{-}011&7.31e\hbox{-}011
	&6.64e\hbox{-}011\\
100&0
	&6.97e\hbox{-}010&6.19e\hbox{-}010&4.71e\hbox{-}010
	&3.16e\hbox{-}010\\
150&0
	&2.05e\hbox{-}009&1.83e\hbox{-}009&1.31e\hbox{-}009
	&1.23e\hbox{-}009\\
200&0
	&4.44e\hbox{-}009&3.12e\hbox{-}009&2.90e\hbox{-}009
	&2.10e\hbox{-}009\\
250&0
	&8.25e\hbox{-}009&6.38e\hbox{-}009&6.38e\hbox{-}009
	&3.96e\hbox{-}009\\
300&0
	&1.62e\hbox{-}008&1.15e\hbox{-}008&9.98e\hbox{-}009
	&8.86e\hbox{-}009\\
\end{array}
  \end{equation*}
  \caption{\label{tbl:qcf_ana:conj_evals}
    The spectra of $\Ps_\Us L_{{\rm qcf},F}$ and $\E_{\rm qnl}''(y_F)$ are computed for
    increasing $N$, for $K=N/2$, for $\phi_F'' = 1$, and for different
    values of $\phi_{2F}''$. The table displays the $\ell^2$ norm (not
    scaled by $\eps$) of the ordered vectors of eigenvalues. The column
    for $\phi''_{2F}=0$ is identically zero since, in this case, the two
    operators coincide.  All other entries are zero to numerical precision
    of the eigenvalue solver.
  }
  \vspace{-7mm}
\end{table}

\begin{conjecture}
  \label{samespec}
  For all $N \geq 4,\ 1 \leq K < N,\text{ and } F > 0,$ the operator
  $\Ps_\Us L_{{\rm qcf},F}$ is diagonalizable and its spectrum is
  identical to the spectrum of $\E''_{\rm qnl}(y_F)$.
\end{conjecture}

\medskip \noindent Since $\E''_{\rm qnl}(y_F)$ is positive if and only
if $A_F > 0$ (cf. Proposition \ref{th:stab_qnl}), the validity of the
conjecture would imply that $L_{{\rm qcf},F}$ has positive real
eigenvalues if and only if $A_F > 0.$

To see how this observation implies dynamical stability
\eqref{eq:qcf_ana:defn_dynstab}, let $V$ denote the matrix whose
columns are the eigenvectors for $\Ps_\Us L_{{\rm qcf}, F}.$ Then $V$
has full rank and $V^{-1} \Ps_\Us L_{{\rm qcf},F} V$ is a diagonal matrix with
the eigenvalues of $L_{{\rm qcf},F}$ on its diagonal. If we define
$z(t) = V^{-1} u(t),$ then
\begin{equation*}
  \begin{split}
    \ddot\zb(t) + V^{-1} \Ps_\Us L_{{\rm qcf},F} V \zb(t) = 0, \\
    \zb(0) = V^{-1} u_0, \quad \zb'(0) = 0.
  \end{split}
\end{equation*}
The solution to the above system of equations is $z_j(t) = z_j(0)
\cos(\sqrt{\lambda_j} t)$ which clearly satisfies the bound $\|z(t)
\|_{\ell^2_\eps} \leq \| V^{-1} u_0\|_{\ell^2_\eps}$ for all $t$.
Thus, we can estimate
\begin{equation*}
  \begin{split}
    \| u(t) \|_{\ell^2_\eps}
    \leq~& \|V\|_{L(\ell^2_\eps,\ell^2_\eps)} \| V^{-1} u(t) \|_{\ell^2_\eps} \\
    \leq~& \|V\|_{L(\ell^2_\eps,\ell^2_\eps)} \| V^{-1} u_0 \|_{\ell^2_\eps} \\
    \leq~& \cond(V) \| u_0\|_{\ell^2_\eps},
  \end{split}
\end{equation*}
where the condition number of $V$ is defined as usual by
$\cond(V)=\|V\|\, \|V^{-1}\|.$ Hence, we see that, subject to the
validity of Conjecture \ref{samespec}, the $L_{{\rm qcf},F}$ operator
satisfies \eqref{eq:qcf_ana:defn_dynstab} with constant $C =
\cond(V)$. To make this stability independent of $N$, we require that
$\cond(V)$ is bounded as $N \to \infty$. This is the subject of
further numerical experiments displayed in Table
\ref{tbl:qcf_ana:condV}. They suggest that this is indeed true if and
only if $A_F > 0$.

\begin{table}[t]
\begin{equation*}
\begin{array}{r|rrrrrrrr}
\textrm{N}&\phi_{2F}''= 0&  -0.1& -0.15&  -0.2&  -0.24\\
\hline
  10&  1.00&1.4563&1.7607&2.3770&5.4398\\
  30&  1.00&1.5242&1.9441&2.8049&6.2876\\
  90&  1.00&1.5794&2.0537&3.0726&7.4324\\
 270&  1.00&1.6049&2.1095&3.1510&8.2878\\
 810&  1.00&1.6136&2.1191&3.2021&8.3863\\
2430&  1.00&1.6139&3.7477&3.2075&8.5968\\
\end{array}
  \end{equation*}
  \caption{\label{tbl:qcf_ana:condV} The table shows the
    condition number $\cond(V)$ for increasing values of $N$, for $K = N/2$, for $\phi_F'' = 1$,
    and for different values of $\phi''_{2F}.$  In each column,
    the computed condition numbers appear to approach an upper bound.
  }
  \vspace{-7mm}
\end{table}

\begin{conjecture}
  \label{uniformcond}
  Let $V$ denote the matrix of eigenvectors for the force-based
  QC operator $\Ps_\Us L_{{\rm qcf},F}.$ If $A_F > 0,$ then
  $\cond(V)$ is uniformly bounded in $N$.
\end{conjecture}

\medskip Conjectures~\ref{samespec} and \ref{uniformcond}, supported
by the results of the numerical experiments that we have presented in
Tables \ref{tbl:qcf_ana:conj_evals} and \ref{tbl:qcf_ana:condV}, imply
that $\Ps_\Us L_{{\rm qcf},F}$ is indeed dynamically stable for
$A_F>0,$ with a stability constant that is uniform in $N$.

\section*{Conclusion}

We propose that a sharp stability analysis of atomistic-to-continuum
coupling methods is an essential ingredient for the evaluation of
their predictive capability, as important as a sharp consistency
analysis. In the present paper, we have established such a sharp
stability analysis for the force-based QC method, for a simple
one-dimensional model problem. We have analyzed three notions of
stability:
\begin{enumerate}
  \ilist
\item {\em Positivity} (coercivity) is generically not satisfied.

\item {\em Operator stability}, uniformly in the size of the atomistic
  system, holds only with an appropriate choice of function spaces. It
  does not hold for several natural choices.

\item {\em Dynamical stability} is satisfied up to the critical
  load. This result is based on the numerical observation that the
  spectra of the QCF and QNL operators coincide.
\end{enumerate}

Positivity and dynamical stability
are equivalent for energy-based methods,
and under suitable conditions and choices of function spaces they imply
operator stability. However,
the fact that the QCF method is non-conservative and gives rise to
non-normal operators, leads to a much richer mathematical structure.

Finally, we stress once again that, while the QCF method is possibly
the simplest force-based multi-physics coupling scheme, we believe
that similar observations can be made for other force-based hybrid
methods, such as FeAt \cite{kohlhoff}, CADD \cite{cadd} or {\em brutal
  force mixing} \cite{hybrid_review}.

\appendix

\section{Proof of Theorem~\ref{th:qcf:stability}: Stability of
  $L_{{\rm qcf}, F}$}
\label{sec:qcf_stab}
Theorem~\ref{th:qcf:stability} states that, if $\phi_F'' + 8
\phi_{2F}''>0$, then $L_{{\rm qcf},F}$ is stable as an operator from
$\Us^{1,\infty}$ to $\Us^{-1,\infty}$, uniformly in $N$.

The proof of this statement uses a variational representation for the
QCF operator, which we derived in \cite{dobs-qcf2}, and which is also
valid for periodic boundary conditions:
\begin{displaymath}
  L_{{\rm qcf},\,F} = \phi_F'' L_1 + \phi_{2F}'' (L_2^{\rm reg}+L_2^{\rm sng}),
\end{displaymath}
where the three operators $L_1,
L_2^{\rm reg}, L_2^{\rm sng} : \Us \to \Us^*$ are given by
\begin{align*}
  \< L_1 u, v\> =~& \< u', v' \>, \\
  \< L_2^{\rm reg} u, v \> =~& \eps\!\!\! \sum_{\ell = -N+1}^{-K} 4 u_\ell' v_\ell'
  + \eps\!\!\! \sum_{\ell = -K+1}^K (u_{\ell-1}'+2u_\ell'+u_{\ell+1}') v_\ell'
  + \eps\!\!\! \sum_{\ell = K+1}^N 4 u_\ell' v_\ell', \\
  \< L_2^{\rm sng} u, v\> =~& (u_{-K+1}' - 2 u_{-K}' + u_{-K-1}') v_{-K}
  - (u_{K+2}' - 2 u_{K+1}' + u_K') v_K.
\end{align*}
We omit the proof of this representation which is a straightforward
summation by parts argument and carries over verbatim from
\cite{dobs-qcf2}. Upon defining
\begin{displaymath}
  \sigma_\ell(u') = \cases{
    \phi_F''u_\ell' + \phi_{2F}''(u_{\ell-1}'+2u_\ell' + u_{\ell+1}'), & \ell = -K+1,\dots,K, \\
    (\phi_F'' + 4 \phi_{2F}'') u_\ell', & \text{otherwise},
  }
\end{displaymath}
as well as
\begin{equation}\label{alpha}
\begin{split}
  \alpha_K(u') =~& \phi_{2F}''(u_{K+2}' - 2 u_{K+1}' + u_{K}'),
  \quad \text{and} \\
  \alpha_{-K}(u') =~& \phi_{2F}''(u_{-K+1}' - 2u_{-K}' + u_{-K-1}'),
  \end{split}
\end{equation}
we can rewrite this representation as
\begin{displaymath}
  \< L_{{\rm qcf},\,F} u, v \>
  = \< \sigma(u'), v' \> + \alpha_{-K}(u') v_{-K} - \alpha_K(u') v_K.
\end{displaymath}

Using the periodic heaviside function $h\in\Us$ given by
\begin{equation}
  \label{eq:heaviside}
  h_\ell = \cases{
    \hphantom{-} \frac12 (1-\eps\ell) -\smfrac{\eps}{4}, & \ell \geq 0, \\
    -  \frac12 (1+\eps\ell) -\smfrac{\eps}{4}, & \ell < 0,}
\end{equation}
and setting $\tilde h_\ell = h_{\ell-1}$, the point evaluation
functional $v \mapsto v_0$, $v \in \Us$, can be represented by
\begin{displaymath}
  v_0 = \< h', v \> = - \< \tilde h, v'\> \qquad
  \text{for all } v \in \Us.
\end{displaymath}
Combining these observations, we obtain the following result.

\begin{lemma}
  \label{th:qcf_rep}
  $L_{{\rm qcf},F}$ can be written as
  \begin{equation}
    \label{qcfvar1}
    \< L_{{\rm qcf},F} u, v \> = \< E_{{\rm qcf},F} u', v'\>
    \qquad \text{for all } u,v \in \Us,
  \end{equation}
  where
  \begin{equation}
    \label{qcfvar2}
    E_{{\rm qcf},F} u'_\ell = \sigma_\ell(u')
    - \alpha_{-K}(u') h_{\ell+K-1}
    + \alpha_K(u') h_{\ell-K-1},
  \end{equation}
  for $\sigma, h,$ and $\alpha_{\pm K}$ as defined above.
\end{lemma}

Even though the variational representations of the Dirichlet case and
the periodic case are the same, we cannot translate the proof for
inf-sup stability that we used in \cite{dobs-qcf2}, as it required a
matrix representation that is unavailable for periodic boundary
conditions. Instead, we will compute a fairly explicit characterization
of $L_{{\rm qcf},\,F}^{-1}$ to estimate its norm directly. It is
most convenient to do so if we define an equivalent norm on
$\Us^{-1,\infty}$. Note that $L_1 : \Us \to \Us^*$ is bijective, and
hence we can define
\begin{displaymath}
  \| g \|_{\tilde \Us^{-1,\infty}} = \| L_1^{-1} g \|_{\Us^{1,\infty}}
  \qquad \text{for } g \in \Us^*.
\end{displaymath}

\begin{lemma}
  \label{th:alt_neg_norm}
  For all $g \in \Us^*,$ it holds that
  \begin{displaymath}
  \frac 12\| g \|_{\tilde \Us^{-1,\infty}}
  \leq \| g \|_{\Us^{-1,\infty}}
  \leq \| g \|_{\tilde \Us^{-1,\infty}}.
  \end{displaymath}
\end{lemma}
\begin{proof}
  Let $z = L_1^{-1} g$, that is,
  \begin{displaymath}
    \< z', v' \> = \< g, v \> \qquad \forall v \in \Us.
  \end{displaymath}
  Taking the supremum over $v$ with $\|v'\|_{\ell^1_\eps} = 1$ we
  obtain the second inequality
  \begin{displaymath}
    \|g\|_{\Us^{-1,\infty}} \leq \|z'\|_{\ell^\infty_\eps} = \|g\|_{\tilde\Us^{-1,\infty}}
  \end{displaymath}
  by H\"{o}lder's inequality.

  The first inequality follows from the fact, which is proved below,
  that
  \begin{equation}\label{sign}
    \smfrac{1}{2}\|z'\|_{\ell^\infty_\eps}
    \leq \sup_{\substack{v \in \Us \\ \|v'\|_{\ell^1_\eps} = 1}}
    \< z', v'\> \qquad \forall z\in\Us.
  \end{equation}
  Namely, this implies that
  \begin{displaymath}
   \smfrac{1}{2}\| g \|_{\tilde \Us^{-1,\infty}}
   = \smfrac{1}{2} \|z'\|_{\ell^\infty_\eps}
   \leq \sup_{\substack{v \in \Us \\ \|v'\|_{\ell^1_\eps} = 1}} \< z', v'\>
   = \sup_{\substack{v \in \Us \\ \|v'\|_{\ell^1_\eps} = 1}} \< g, v \>
   = \|g\|_{\Us^{-1,\infty}}.
  \end{displaymath}

  To prove \eqref{sign}, we fix $z\in\Us$ and let $\ell_1, \ell_2$ be
  such that $z'_{\ell_1}=\|z'\|_{\ell^\infty_\eps}$ and
  $z'_{\ell_2}<0$. (A similar argument can be used if
  $z'_{\ell_1}=-\|z'\|_{\ell^\infty_\eps}$.) We obtain \eqref{sign}
  from the fact that $\smfrac12 \|z'\|_{\ell^\infty_\eps} \leq \< z', v'
  \>$ where $v\in\Us$ is defined by
  \begin{displaymath}
    v'_\ell=\cases{
      \smfrac 1{2\eps}&\text{ if }\ell=\ell_1,\\
      -\smfrac 1{2\eps}&\text{ if }\ell=\ell_2,\\
      0&\text{ otherwise.}
    }
    \qedhere
  \end{displaymath}
\end{proof}

\begin{corollary}
  \label{th:Lqcfinv_and_L1invLqcf}
  Suppose that $F$ is such that $L_{{\rm qcf}, F} : \Us \to \Us^*$ is
  invertible, then
  \begin{displaymath}
    \begin{split}
      \| (L_1^{-1} L_{{\rm qcf},F})^{-1} \|_{L(\Us^{1,\infty},\ \Us^{1,\infty})}
      \leq~& \| L_{{\rm qcf},F}^{-1} \|_{L(\Us^{-1,\infty},\ \Us^{1,\infty})}  \\
      \leq~& 2 \| (L_1^{-1} L_{{\rm qcf},F})^{-1} \|_{
        L(\Us^{1,\infty}, \ \Us^{1, \infty})}.
    \end{split}
  \end{displaymath}
\end{corollary}
\begin{proof}
  Using Lemma \ref{th:alt_neg_norm} twice, we can prove the following
  bound,
  \begin{align*}
    \frac{1}{
      \| (L_1^{-1} L_{{\rm qcf}, F})^{-1} \|_{L(\Us^{1,\infty},\ \Us^{1,\infty})}}
    =~& \!\inf_{\substack{u \in \Us \\ \|u\|_{\Us^{1,\infty}} = 1}} \!
    \| L_1^{-1} L_{{\rm qcf}, F} u \|_{\Us^{1,\infty}}
    = \!\inf_{\substack{u \in \Us \\ \|u\|_{\Us^{1,\infty}} = 1}} \!
    \| L_{{\rm qcf}, F} u \|_{\tilde \Us^{-1,\infty}} \\
    \geq~& \!\inf_{\substack{u \in \Us \\ \|u\|_{\Us^{1,\infty}} = 1}} \!
    \| L_1^{-1} L_{{\rm qcf}, F} u \|_{\Us^{1,\infty}}
    =     \frac{1}{
      \| L_{{\rm qcf}, F}^{-1} \|_{L(\Us^{-1,\infty},\ \Us^{1,\infty})}},
  \end{align*}
  which gives the first stated inequality. The second inequality
  follows from a similar argument.
\end{proof}

Corollary \ref{th:Lqcfinv_and_L1invLqcf} shows that we can bound the
operator norm $\| L_{{\rm qcf},\,F}^{-1} \|_{L(\Us^{-1,\infty},\
  \Us^{1,\infty})}$ in terms of $\| (L_1^{-1} L_{{\rm qcf},\,F})^{-1}
\|_{L(\Us^{1,\infty},\ \Us^{1,\infty})}$. The latter operator norm can
be computed using the formula
\begin{equation}
  \label{eq:L1invLqcf_inv}
  \| (L_1^{-1} L_{{\rm qcf}, F})^{-1} \|_{L(\Us^{1,\infty},\ \Us^{1,\infty})}
  = \Big\{\inf_{\substack{u \in \Us \\ \|u'\|_{\ell^\infty_\eps} = 1}}
  \| (L_1^{-1}L_{{\rm qcf}, F} u)' \|_{\ell^\infty_\eps} \Big\}^{-1}.
\end{equation}
In the next lemma, we establish an explicit representation of
$L_1^{-1} L_{{\rm qcf},\,F}$ which will subsequently allow us to
construct upper and lower bounds for \eqref{eq:L1invLqcf_inv}.

\begin{lemma}
  \label{th:qppqcf:lemma_z}
  Let $z = L_1^{-1} L_{{\rm qcf},\,F} u$, then
  \begin{align*}
    z_\ell' =~& \sigma_\ell(u')
    - \frac{\eps}{2} \phi_{2F}'' \big\{ u_{-K}' - u_{-K+1}'
    - u_{K}' + u_{K+1}' \big\} \\
    \notag
    & - \alpha_{-K}(u') h_{\ell+K-1} + \alpha_K(u') h_{\ell-K-1},
  \end{align*}
  where $\sigma, h,$ and $\alpha_{\pm K}$ are defined above.
\end{lemma}

\begin{remark}
  We note that the term $\frac12\eps \{ u_{-K}' - u_{-K+1}' - u_K' +
  u_{K+1}'\}$ is the average of $\sigma$, and the function $h$ is a
  periodic heaviside function defined in \eqref{eq:heaviside}.
\end{remark}

\begin{proof}[Proof of Lemma \ref{th:qppqcf:lemma_z}]
  The function $z$ is the solution of the variational
  principle
  \begin{displaymath}
    \< z', v' \> = \< L_{{\rm qcf},\,F} u, v\>
    = \< E_{{\rm qcf},F} u', v' \>
  \end{displaymath}
  where $E_{{\rm qcf}, F}$ is defined in \eqref{qcfvar2}, and is given
  by
  \begin{equation*}
    E_{{\rm qcf},F} u'_\ell = \sigma_\ell(u') - \alpha_{-K}(u') h_{\ell+K-1}
    + \alpha_K(u') h_{\ell-K-1}.
  \end{equation*}
  We note that a function $w \in \R^{2N}$ is a gradient, that is, $w =
  v'$ for some $v \in \Us$, if and only if $\sum_{\ell = -N+1}^N
  w_\ell = 0$. Hence, we obtain $z' = E_{{\rm qcf},F} u' -
  \overline{E_{{\rm qcf},F} u'}$ where $\overline{E_{{\rm qcf},F} u'}
  := \frac12 \eps \sum_{\ell = -N+1}^N E_{{\rm qcf},F} u'_\ell$. Since
  $h$ has zero mean, we only need to compute $\bar\sigma$,
  \begin{displaymath}
    \bar\sigma := \frac{1}{2N} \sum_{\ell = -N+1}^N \sigma_\ell \\
    = \frac{A_F}{2N} \sum_{\ell = -N+1}^N  u_\ell'
    + \frac{\phi_{2F}''}{2N} \sum_{\ell = -K+1}^K (u_{\ell-1}' - 2 u_\ell'
    + u_{\ell+1}').
  \end{displaymath}
  Since $u$ is periodic, $u'$ has zero mean, and hence the first sum
  on the right-hand side vanishes. The second sum has telescope
  structure, and we obtain
  \begin{displaymath}
    \overline{E_{{\rm qcf},F} u'} = \bar\sigma
    = \frac{\eps}{2} \phi_{2F}'' ( u_{-K}' - u_{-K+1}' - u_K' + u_{K+1}' ).
  \end{displaymath}
  This concludes the proof of the lemma.
\end{proof}

We are now ready to conclude the proof of Theorem
\ref{th:qcf:stability}.

\begin{proof}[Proof of Theorem \ref{th:qcf:stability}]
We set $z = L_1^{-1} L_{{\rm qcf},\,F} u$ and use Lemma
\ref{th:qppqcf:lemma_z} to deduce the bound
\begin{equation}
  \label{eq:qcf_stab_proof_1}
  \begin{split}
    \|z'\|_{\ell^\infty_\eps} \geq~& \|\sigma(u')\|_{\ell^\infty_\eps}
    - 2 \eps |\phi_{2F}''| \|u'\|_{\ell^\infty_\eps} \\
    & - \max(|\alpha_{-K}(u')|, |\alpha_K(u')|) \max_{\ell}
    (|h_{\ell+K-1}| + |h_{\ell-K-1}|).
  \end{split}
\end{equation}
To bound the first term on the right-hand side, we note that
\begin{displaymath}
  |\sigma_\ell(u')| \geq \phi_F'' |u_\ell'| + 4 \phi_{2F}'' \|u'\|_{\ell^\infty_\eps},
\end{displaymath}
which immediately implies
\begin{equation}
  \label{eq:a_nice_formula}   
  \|\sigma(u')\|_{\ell^\infty_\eps} \geq A_F \|u'\|_{\ell^\infty_\eps}.
\end{equation}
To bound the third term on the right-hand side of
\eqref{eq:qcf_stab_proof_1}, we crudely estimate
\begin{displaymath}
  \max_{\ell = -N+1, \dots, N} (|h_{\ell+K-1}| + |h_{\ell-K-1}|)
  \leq 1 - \smfrac12 \eps,
\end{displaymath}
which is true whenever $K \geq 1,$ and deduce from \eqref{alpha} that
\begin{displaymath}
\max(|\alpha_{-K}(u')|, |\alpha_K(u')|)\le 4|\phi_{2F}''|.
\end{displaymath}
The additional term $-\frac12 \eps$
cancels with the second term on the right-hand side of
\eqref{eq:qcf_stab_proof_1}, so that we obtain
\begin{displaymath}
  \|z'\|_{\ell^\infty_\eps} \geq (\phi_F'' + 8 \phi_{2F}'') \| u' \|_{\ell^\infty_\eps}.
\end{displaymath}
Employing Corollary \ref{th:Lqcfinv_and_L1invLqcf} and Formula
\eqref{eq:L1invLqcf_inv}, we obtain Theorem
\ref{th:qcf:stability}.
\end{proof}

\section{Proof of Theorem \ref{th:qcf:infsup}: Instability of $L_{{\rm
      qcf}, F}$}
\label{p}

We now prove Theorem~\ref{th:qcf:infsup} on the instability of
$L_{{\rm qcf},F}$ as an operator acting between $\Us^{1,p}$ and
$\Us^{-1,p}$, $1 \leq p < \infty$.  The bound $\| L_{{\rm qcf},F}^{-1}
\|_{L(\Us^{-1,p}, \Us^{1,p})} \geq C N^{1/p}$ follows from the
following lemma.

\begin{lemma}
Suppose that $\phi_F''>0$, $\phi_{2F}'' \in \R \setminus \{0\},$
and $p,q \in \R$ satisfy
$1 \leq p < \infty,$ $1 < q \leq \infty,$ and
$\frac{1}{p} + \frac{1}{q} = 1.$ Then there exists a constant
$C > 0$ such that
\begin{displaymath}
  \inf_{\substack{\vb\in\Us \\ \lpnorm{\vb'}{p} = 1}} \
  \sup_{\substack{\wb\in\Us \\ \lpnorm{\wb'}{q} = 1}}
  \<  L_{{\rm qcf},F} \vb,\,\wb \>
  \le C N^{-1/p}.
\end{displaymath}
\end{lemma}

\begin{proof}
  We recall from Lemma~\ref{th:qppqcf:lemma_z} that we can represent
  $L_{{\rm qcf},F} v$ in the form
  \begin{equation}
  \label{qcfvar3}
    \< L_{{\rm qcf},F} v, w \> = \< E_{{\rm qcf},F} v', w'\> \qquad \forall w \in \Us,
  \end{equation}
  where
  \begin{equation*}
    E_{{\rm qcf}, F} v'_\ell = \sigma_\ell(v') - \alpha_{-K}(v') h_{\ell+K-1}
    + \alpha_K(v') h_{\ell-K-1},
  \end{equation*}
  and where
  \begin{align*}
    \sigma_\ell(v') &= \cases{
      \phi_F''v_\ell' + \phi_{2F}''(v_{\ell-1}'+2v_\ell' + v_{\ell+1}'),
      & \ell = -K+1,\dots,K, \\
      (\phi_F'' + 4 \phi_{2F}'') v_\ell', & \text{otherwise},
    }\\
    \alpha_K(v') &= \phi_{2F}''(v_{K+2}' - 2 v_{K+1}' + v_{K}'), \\
    \alpha_{-K}(v') &= \phi_{2F}''(v_{-K+1}' - 2v_{-K}' + v_{-K-1}'),\\
    h_\ell &= \cases{
      \hphantom{-} \frac12 (1-\eps\ell)- \smfrac{\eps}{4},
      & \ell \geq 0, \\
      - \frac12 (1+\eps\ell)-\smfrac{\eps}{4}, & \ell < 0.
    }
  \end{align*}
  We choose $v \in \Us$ with derivative given by
  \begin{equation*}
    v'_\ell = \cases{
      0,& \ell = K-1,\\
      -\frac{A_F}{6 \phi''_{2F}},& \ell = K,\\
      \hphantom{-} \frac{A_F}{3 \phi''_{2F}},& \ell = K+1,\\
      -\frac{A_F}{6 \phi''_{2F}},& \ell = K+2,\\
      h_{\ell-K-1}, & \text{otherwise}.
    }
  \end{equation*}
  Such a representation is possible if and only if the vector
  $(v_\ell')_{\ell = -N+1}^N$ defined above has zero mean. To see that
  this holds, we use the symmetry of $h_\ell$ to calculate
  \begin{equation*}
    \sum_{\ell=-N+1}^N v'_\ell
    = \sum_{\ell \neq K-1,\,K,\,K+1,\,K+2} h_{\ell-K-1} = 0.
  \end{equation*}

  If we insert $v$ into the equations above, we find that
  \begin{displaymath}
    \alpha_{-K}(v') = 0, \quad \alpha_K(v') = - A_F,
  \end{displaymath}
  and
  \begin{equation*}
    \sigma_\ell(v') = \cases{
      (\phi''_F + 2\phi''_{2F}) h_{-3} + \phi''_{2F} h_{-4},& \ell = K-2,\\[6pt]
      \phi''_{2F} h_{-3} - \smfrac{1}{6} A_F, & \ell = K-1,\\[6pt]
      -\frac{\phi''_F A_F}{6 \phi''_{2F}},
      & \ell = K,\\[6pt]
      \hphantom{-}
      \frac{A_F^2}{3 \phi''_{2F}}, & \ell = K+1,\\[6pt]
      -\frac{A_F^2}{6 \phi''_{2F}}, & \ell = K+2,\\[6pt]
      A_Fh_{\ell-K-1} , & \text{otherwise},\\[6pt]
    }
  \end{equation*}
  which implies that
  \begin{equation*}
    \begin{split}
      E_{{\rm qcf},F} v'_\ell = \cases{
        - 2\phi''_{2F} h_{-3} + \phi''_{2F} h_{-4},& \ell = K-2,\\[6pt]
        \phi''_{2F} h_{-3} - \smfrac{1}{6} A_F- A_F h_{-2}, & \ell = K-1,\\[6pt]
        -\frac{\phi''_F A_F}{6 \phi''_{2F}}- A_F h_{-1},
        & \ell = K,\\[6pt]
        \hphantom{-}
        \frac{A_F^2}{3 \phi''_{2F}} - A_F h_{0}, & \ell = K+1,\\[6pt]
        -\frac{A_F^2}{6 \phi''_{2F}} - A_F h_{1}, & \ell = K+2,\\[6pt]
        0, & \text{otherwise}.\\[6pt]
      }
    \end{split}
  \end{equation*}
  Note that all the terms above are bounded in absolute value, independently
  of $N$ and $K.$

  Inserting these formulas into~\eqref{qcfvar3}, applying H\"older's
  inequality, and using the fact that $E_{{\rm qcf},F} v'_\ell$ is nonzero for only
  five indices, we obtain
  \begin{equation*}
    \begin{split}
      \< L_{{\rm qcf},F} v, w\> &= \< E_{{\rm qcf},F} v', w'\> \\
      &\leq \lpnorm{E_{{\rm qcf},F} v'}{p} \lpnorm{w'}{q} \\
      &\leq \eps^{1/p} \Big[
      5 \lpnorm{E_{{\rm qcf},F} v'}{\infty}^p
      \Big]^{1/p} \lpnorm{w'}{q} \\
      &\leq C \eps^{1/p}\lpnorm{w'}{q}.
    \end{split}
  \end{equation*}

  It remains to show that $\|v'\|_{\ell^p_\eps}$ is bounded below as
  $N \to \infty$. As a matter of fact, it can be seen from the
  definition of $v_\ell'$ that
  \begin{displaymath}
    |v'_{\ell}| \ge \smfrac{1}{4} \qquad
    \text{for } j = K+1-N/2,\dots,K-1,
  \end{displaymath}
  which gives
  \begin{displaymath}
    \lpnorm{v'}{p} \geq \Bigg[\sum_{\ell=K+1-N/2}^{K-1} \eps
    \big(\smfrac{1}{4}\big)^p \Bigg]^{1/p}
    = \smfrac{1}{4} \big[(N/2-2) \eps\big]^{1/p}.
  \end{displaymath}
  Thus, replacing $v$ by $v / \lpnorm{v'}{p}$ gives the desired
  result.
\end{proof}

\end{document}